\documentclass{article}
\usepackage{amsmath}
\usepackage{amssymb}
\usepackage{latexsym}
\usepackage{url}

\usepackage[all]{xy}
\SelectTips{cm}{}
\newdir{ >}{{}*!/-7pt/@{>}}
\entrymodifiers={+!!<0pt,\fontdimen22\textfont2>}
\newcommand{\ulpullback}[1][ul]{\save*!/#1-1.2pc/#1:(-1,1)@^{|-}\restore}

\newcommand{\drpullback}[1][dr]{\save*!/#1-1.2pc/#1:(-1,1)@^{|-}\restore}

\usepackage{theorem}
\theorembodyfont{\normalfont\itshape}
\theoremstyle{change}
\newtheorem{lemma}{Lemma.}[section]
\newtheorem{prop}[lemma]{Proposition.}
\newtheorem{theorem}[lemma]{Theorem.}

\theorembodyfont{\normalfont}

\newtheorem{taller}[lemma]{$\!\!$}

\newenvironment{blanko}[1]%
{\begin{taller}{\normalfont\bfseries #1}\normalfont}%
{\end{taller}}

\providecommand{\qed}{\hspace*{\fill}$\Box$}

\newenvironment{proof}{%
\begin{list}{\em Proof. }%
{\setlength{\labelsep}{0mm}\setlength{\leftmargin}{0mm}%
\setlength{\labelwidth}{0mm}\setlength{\listparindent}{\parindent}%
\setlength{\parsep}{\parskip}\setlength{\partopsep}{0mm}}%
\item}{\qed\end{list}}

\newenvironment{proof*}[1]{%
\begin{list}{\em #1 }%
{\setlength{\labelsep}{0mm}\setlength{\leftmargin}{0mm}%
\setlength{\labelwidth}{0mm}\setlength{\listparindent}{\parindent}%
\setlength{\parsep}{\parskip}\setlength{\partopsep}{0mm}}%
\item}{\qed\end{list}}

\usepackage{mathrsfs}

\newcommand{\ov}{\overline}

\newcommand{\op}{^{\text{{\rm{op}}}}}
\providecommand{\kat}[1]{\text{\textbf{\textsl{#1}}}}

\newcommand{\el}{\operatorname{el}}
\newcommand{\iso}{\operatorname{iso}}
\newcommand{\Sub}{\operatorname{Sub}}

\newcommand{\isopil}{\stackrel{\raisebox{0.1ex}[0ex][0ex]{\(\sim\)}}%
			{\raisebox{-0.15ex}[0.28ex]{\(\rightarrow\)}}}

\newcommand{\Aut}{\operatorname{Aut}}
\DeclareMathOperator*{\colim}{colim}

\newcommand{\Hom}{\operatorname{Hom}}

\newcommand{\Set}{\kat{Set}}
\newcommand{\N}{\mathbb{N}}

\newcommand{\Sh}{\kat{Sh}}
\newcommand{\PrSh}{\kat{PrSh}}
\newcommand{\Gr}{\kat{Gr}}

\newcommand{\Grwie}{\kat{G\.r}}
\newcommand{\elGr}{\kat{elGr}}
\newcommand{\graf}[1]{\mathsf{#1}}
\newcommand{\triv}{\graf{I}}

\newcommand{\comma}{\raisebox{1pt}{$\downarrow$}}

\renewcommand{\epsilon}{\varepsilon}

\usepackage{texdraw}
\input{txdtools}

\everytexdraw{%
	\drawdim pt \linewd 0.5 \textref h:C v:C
	\setunitscale 1
    \arrowheadtype t:F
    \arrowheadsize l:5 w:3
}

\newcommand{\onedot}{
  \bsegment
    \move (0 0) \fcir f:0 r:2
  \esegment
}

\makeatletter
\newcommand{\subjclass}[2][2010]{%
  \let\@oldtitle\@title%
  \gdef\@title{\@oldtitle\footnotetext{#1 \emph{Mathematics subject classification.} #2}}%
}
\newcommand{\keywords}[1]{%
  \let\@@oldtitle\@title%
  \gdef\@title{\@@oldtitle\footnotetext{\emph{Key words and phrases.} #1.}}%
}
\makeatother

\begin{document}

\title{Cospan construction of the graph category 
% \\[6pt] 
of Borisov and Manin}

\author{\sc Joachim Kock}

\date{\footnotesize \em To Nils Baas, on his 70th birthday}


\maketitle

\begin{abstract}
  It is shown how the graph category of Borisov and Manin can be constructed
  from (a variant of) the graph category of Joyal and Kock, essentially by
  reversing the generic morphisms.
  More precisely, the morphisms in the Borisov--Manin category are exhibited as
  cospans of reduced covers and refinement morphisms.
\end{abstract}

%%%%%%%%%%%%%%%%%%%%%%%%%%%%%%%%%%%%%%%%%%%%%%%%%%
\section{Introduction}
%%%%%%%%%%%%%%%%%%%%%%%%%%%%%%%%%%%%%%%%%%%%%%%%%%

The following definition of graph with open-ended edges is 
quite standard in the literature, and is usually attributed 
to Kontsevich and Manin~\cite{KM:9402}: a graph is a quadruple $(V,F,\partial,j)$ where
$V$ is a set of {\em vertices}, $F$ is a set of {\em flags} (also called half-edges),
$\partial : F \to V$ is a map (assigning to a flag the vertex it originates from)
and $j: F \to F$ is an involution.  The idea is that the involution interchanges
two flags if they are attached to each other to form an {\em edge}.  The fixpoints
for $j$ are interpreted as open-ended edges, called {\em tails} (this set is denoted 
$T$).

The `problem' with this definition (due to the fact that tails are encoded as 
fixpoints, although a tail is not a flag glued to itself) 
is that it does not naturally lead to good
notions of morphisms, beyond isomorphisms: naive attempts in terms of
structure-preserving maps tend to preserve too much structure.

Borisov and Manin~\cite{Borisov-Manin:0609748} overcame this problem with a
rather intricate definition of morphism, reproduced below.
Their category is rich enough to capture the combinatorics of both operad-like 
and prop-like structures.
Although it works, it is perhaps fair to say that it is not a very
straightforward or natural-looking definition, and that it takes many examples
and drawings to grasp its meaning.

A main point to make to explain the significance of the notion and the intuition
behind it is that there are two main classes of morphisms, called graftings and
compressions  (the compressions are spanned by the contractions and mergers of 
Borisov--Manin), each expressing operations that are fundamental to graph theory: a
grafting is a morphism that takes a graph (connected or not) and constructs a
new graph by clutching together some of its tails to form edges.  A compression
is a morphism that takes a graph, chooses a collection of nonempty (but not
necessarily connected) subgraphs, and contracts each of them to a vertex.  The
point of the Borisov--Manin category is that every morphism factors as a
grafting followed by a compression, and that this factorisation is essentially
unique (cf.~Proposition~\ref{prop:fact} below).
In this way their category can be seen to be the smallest category
containing these two classes of morphisms and a sensible notion of composition of
such morphisms.  It is remarkable that an elementary description could be found to
achieve this.

\bigskip

The present contribution takes the factorisation property
as the starting point for a description of the category, instead of the
intricate elementary description.  This can be formalised through a general
procedure for constructing a category from two given classes of morphisms, $G$ 
and $C$, with a common object set.  The new morphisms are declared to be 
formal pairs consisting of a morphism in $G$ followed by a morphism in $C$.
To define the composition law for such pairs, one needs a rule for transforming
a pair in the wrong order into a pair in the right order, subject to some
axioms reminiscent of the axioms for bimodules or braids.  This is called a
commutation law, or sometimes a distributive law~\cite{Rosebrugh-Wood:2002}.

In this strategy, the question then becomes first to describe the two
classes of morphisms individually, hopefully in a more conceptual way, and then to
specify the commutation law.  Most of this has already been accomplished, namely in the Joyal--Kock formalism of
graphs~\cite{Joyal-Kock:0908.2675}, recalled below.  This formalism starts with 
a definition of graph featuring a canonical notion of {\em etale morphism}
(such as open inclusions).  This defines the etale topology and a
notion of etale cover; a cover is {\em reduced} when it is bijective on
vertices.  A reduced cover is essentially the same thing as a grafting in the 
sense of Borisov-Manin (cf.~Proposition~\ref{prop:JKrc=BMgr}).

The second class of morphisms is generated from the first by a monad, a version
of the free-compact-symmetric-multicategory monad~\cite{Joyal-Kock:0908.2675}, 
which could be called the free-compact-coloured-prop monad: this
is the class of {\em generic morphisms}, a general categorical
notion~\cite{Weber:TAC13}, which in this case can be described in elementary
terms as {\em refinements}.  A refinement is the opposite of a compression
(Proposition~\ref{prop:JKgen=BMcomp}).  Note here that since a compression is
allowed to contract a non-connected graph, an important minor adjustment is required
in the description of the monad compared to \cite{Joyal-Kock:0908.2675} 
(essentially passing from connected to non-connected).

These two classes of morphisms already live in a single bigger category, namely the
Kleisli category for the monad~\cite{Joyal-Kock:0908.2675},
but in here they interact differently since a
refinement morphism between two graphs in the Kleisli category goes in the opposite
direction of the corresponding compression morphism in the BM category.  However, the
key observation at this point is that reduced covers and refinement morphisms admit
pushouts along each other, yielding new morphisms in the same classes 
(Proposition~\ref{prop:pushout}).  The pushout 
construction
is precisely the required commutation law---in this case a prescription for
formally reversing one of the two classes: the new morphisms
from graph $\graf{T}$ to graph $\graf{R}$ are the cospans in which the 
first leg is a reduced
cover and the second leg a refinement morphism:
$$\xymatrix @C+1.3em {
   \graf{T}\phantom{|}\ar[r]^-{\text{red.cover}} & \phantom{|}\graf{S}\phantom{|} & 
   \ar[l]_-{\text{refinement}} \phantom{|}\graf{R}   .
}$$
The composition of cospans is given by pushout.  Now the main result is this:

\bigskip

\noindent
{\bf Theorem \ref{thm:main}.} {\em The Borisov--Manin category of graphs is equivalent to
the category whose morphisms are cospans of reduced covers and refinement morphisms
in the Kleisli category.}

\bigskip

The result shows that the 
Borisov--Manin category of graphs, although at first sight it may appear a bit
ad hoc, can be derived from general principles and standard constructions:
indeed the notions of etale morphism and reduced cover are completely canonical,
and the notion of generic morphism is characterised by a universal property,
relative to the monad, in turn defined in terms of colimits and presheaves.
This gives some theoretical justification for the BM definition, complementing
the significant empirical fact that it just works, and complementing its
pleasing elementary character.
% As such it answers one instance of a famous generic question of
% Martin Hyland: {\em That's all well and good in practice, but how does it work 
% in theory?}

Although the main arguments in the proof are conceptual, some details of the 
comparison must be done by hand.  These calculations are illustrative for both 
formalisms, and it is hoped that they too can be of some interest also beyond the 
theorem established.

%%%%%%%%%%%%%%%%%%%%%%%%%%%%%%%%%%%%%%%%%%%%%%%%%%
\section{Borisov--Manin category of graphs}
%%%%%%%%%%%%%%%%%%%%%%%%%%%%%%%%%%%%%%%%%%%%%%%%%%

With graphs defined as in the first paragraph of the
introduction, the graph category $\kat{BM}$ of
Borisov and Manin~\cite{Borisov-Manin:0609748} has the following morphisms.
\begin{quote}
{\bf Definition \cite[1.2.1.]{Borisov-Manin:0609748}.}
Let $\tau$, $\sigma$ be two graphs.
{\it A morphism $h:\,\tau\to \sigma$
is a triple $(h^F,h_V, j_h)$, where
$h^F:\,F_{\sigma}\to F_{\tau}$ is a contravariant map,
$h_V:\,V_{\tau}\to V_{\sigma}$ is a covariant map,
and $j_h$ is an involution
on the set $F_{\tau}\setminus h^F(F_{\sigma})$. This data
must satisfy the following conditions.

\smallskip

(i) $h^F$ is injective, $h_V$ is surjective.

\smallskip

(ii) The image $h^F(F_{\sigma})$ and its complement $F_{\tau}\setminus h^F(F_{\sigma})$
are $j_{\tau}$--invariant subsets of flags.

\smallskip

(ii')
The involution $j_h$ is fixpoint free, and agrees with $j_\tau$ on edges
(but must necessarily disagree on tails, since $j_h$ is fixpoint free).

% \smallskip
% 
% It will be convenient to extend $j_h$ to other flags in
% $F_{\tau}$ by identity.
% 

\smallskip

We will say that $h$ contracts all flags in
$F_{\tau}\setminus h^F(F_{\sigma})$.
If two flags in $F_{\tau}\setminus h^F(F_{\sigma})$
form an edge, we say that this edge is contracted by $h$.
If two tails in $F_{\tau}\setminus h^F(F_{\sigma})$
form an orbit of $j_h$, we say that
it is a virtual edge contracted by $h$.

\smallskip

(iii) If a flag $f_{\tau}$ is not contracted by $h$,
that is, has the form $h^F(f_{\sigma})$, then
$h_V$ sends $\partial_{\tau}f_{\tau}$ to
$\partial_{\sigma}f_{\sigma}$. Two vertices
of a contracted edge (actual or virtual) must have the same $h_V$--image.

\smallskip

(iv) The bijection $h_F^{-1}:\, h^F(F_{\sigma})\to F_{\sigma}$
maps edges of $\tau$ to edges of $\sigma$.

If it maps a pair
of tails of $\tau$ to an edge of $\sigma$,
we will say that $h$ grafts these tails.

\smallskip

Given two morphisms
$$\xymatrix{
   \tau \ar[r]^g & \sigma \ar[r]^c & \rho
}$$
the composite morphism $h= c \circ g$ is defined to have
$h_V = c_V \circ g_V$  and $h^F = g^F \circ c^F$.  To complete the definition 
of $h$ we must  define the fixpoint-free involution
$j_h$ on the complement of $h^F$, which is the (disjoint) union of the 
complement of $g^F$ in $F_\tau$ and the complement of $c^F$ in $F_\sigma\subset 
F_\tau$.   We define $j_h$
to be $j_g$ on the first and $j_c$ on the second.
}
\end{quote}

  The definition is reproduced almost verbatim from
  \cite{Borisov-Manin:0609748}, only with a slight adjustment to the definition
  of $j_h$, here defined on the whole complement. In \cite{Borisov-Manin:0609748}
  it is defined only on the set of
  tails in the complement.  This adjustment 
  facilitates the description of
%   seems to be required for the
  the composition law  (adjusted accordingly).

% \tiny
% 
% \begin{lemma}\label{lem:axiii}
%   For a morphism of graphs $h: \tau\to \sigma$ we have the following 
%   commutative diagram 
%   $$\xymatrix{
%      F_\tau \ar[d]_{\partial} & F_\sigma \ar[l]_{h^F}\ar[d]^\partial\\
%      V_\tau \ar[r]_{h_V} & V_\sigma
%   }$$
% \end{lemma}
% \begin{proof}
%   Reformulation of Axiom (iii).
% \end{proof}
% 
% \normalsize

\begin{blanko}{Grafting.}
  A morphism of graphs $h:\tau\to\sigma$ 
  is called a {\em grafting} when both $h_V$ and $h^F$ are
  bijective.
\end{blanko}

\begin{blanko}{Compression.}
  A morphism of graphs $h:\tau\to\sigma$  is called a {\em compression} when $h^F$ is bijective on
  tails.
\end{blanko}

  The notion of grafting is verbatim from \cite{Borisov-Manin:0609748}.  The
  class of compressions is spanned by two classes of morphisms described in
  \cite{Borisov-Manin:0609748}: contractions and mergers.
  A {\em contraction} is a compression for which
  any two vertices with the same image under $h_V$ are connected by a path of
  contracted edges in $\tau$.
%   This means that a connected subgraph is chosen
%   and contracted.
% %   (Subgraph in the sense of etale morphism injective on vertices.)
  A {\em merger} is a compression for which $h^F$ is also bijective on edges.
%   This means that the morphism does not change any tails or edges at all.  All it
%   does is to merge some vertices.
  The reason for singling out the mergers is that in many applications, this
  class of morphisms is actually excluded.  It is excluded in operad-like situations
  (operads, properads, cyclic operads, modular operads, etc.), whereas it is
  needed in prop-like situations (props, wheeled props, etc.).  The reason 
  presently for treating the two classes together is Proposition~\ref{prop:fact} 
  below.

\begin{lemma}\label{lem:jimplied}
  A grafting or compression $h$ is uniquely determined by $h_V$ and $h^F$
  (i.e.~$j_h$ is implied).
\end{lemma}
\begin{proof}
  Indeed, for a grafting, $h^F$ is bijective, so there is no complement to
  define $j_h$ on.  For a compression, the complement consists entirely of edges
  (since $h^F$ is a bijection on tails), and here we are forced to let $j_h$
  agree with $j_\tau$ (by axiom (ii')).
\end{proof}

\begin{lemma}\label{lem:hj=jh}
  For a compression $h: \sigma\to\rho$, the flag map $h^F :F_\rho \to F_\sigma$
  sends edges to edges, tails to tails, and in particular
  commutes with the involutions $j_\rho$ and $j_\sigma$.
\end{lemma}
\begin{proof}
  Since $h^F$ induces a bijection on tails, it must also map non-tails to
  non-tails.  On the other hand, by axiom (iv) the inverse map $(h^F)^{-1} :
  h^F(F_\rho) \to F_\rho$ maps edges to edges, so it follows that $h^F$ maps
  edges to edges.  Thus altogether it maps tails to tails and edges to edges,
  and hence commutes with the involutions.
\end{proof}

The following lemma expresses the important property that the `subgraph' 
compressed by a compression has the same `interface' (to be formalised in 
\ref{def}) as the vertex it is 
compressed to.  The formulation is slightly awkward due to the fact that
the notion of subgraph has not been defined (and at any rate is not given by
a morphism in the category):
% and because those half-edges that are to be interpreted as 
% tails of the subgraphs are not tails of the ambient graph:
\begin{lemma}\label{lem:boundarypres}
  For a compression $h:\sigma\to\rho$, and for each vertex $x\in V_\rho$, the 
  flag map $h^F: F_\rho\to F_\sigma$  restricts to a bijection
  $$
  \partial^{-1}(x) \isopil h^F(F_\rho) \cap \partial^{-1} h_V^{-1}(x) .
  $$
\end{lemma}
\begin{proof}
  Injectivity follows from Axiom (i).  For surjectivity, note that
%   follow from Lemma~\ref{lem:axiii}: 
  a $\sigma$-flag on the right-hand side must come from 
  some
  $\rho$-flag, and now it follows from Axiom~(iii)
%   Lemma~\ref{lem:axiii} says 
  that this flag is incident to $x$.
\end{proof}

\begin{prop}\label{prop:fact}
  Every morphism of graphs $h: \tau \to \rho$ factors essentially uniquely as
  $$\xymatrix{
     \tau \ar[rr]^h\ar[rd]_g && \rho \\
     & \sigma \ar[ru]_c & 
  }$$
  where $g$ is a grafting and $c$ is a compression.  More precisely, the classes
  of graftings and compressions form a factorisation system in the
  Borisov--Manin category of graphs.
\end{prop}
\begin{proof}
  We construct the middle object $\sigma$ in the factorisation.  Since we want a
  grafting $\tau\to\sigma$, we are forced to take as vertices and flags of
  $\sigma$ the same as those of $\tau$.  Since now $g_V$ and $g^F$ are
  bijections, this also determines $c_V := h_V \circ g_V^{-1}$ and 
  $c^F := (g^F)^{-1} \circ h^F$.  To specify $\sigma$, it only remains to choose
  the involution $j_\sigma$.
%   and check that it is compatible with $j_g$ and $j_c$, in turn
%   uniquely determined by the vertex and flag data.  We claim that $j_\sigma$ is
%   uniquely determined by the requirement that the second morphism $c$ is a compression
%   and by the requirement that the composite of the two morphisms $g$ and $c$ is equal
%   to $h$.   Indeed,
  By Lemma~\ref{lem:hj=jh}, $j_\sigma$ must agree with $j_\rho$ on the image of
  $c^F$ (which is also the image of $h^F$).
%   Outside this image we need to pair 
%   up all flags, because there are no more tails in $\sigma$ than in $\rho$ 
%   (again since $c$ is a compression). 
%   But there is only one 
%   way
%   to pair them up in such a way that $c \circ g = h$, namely by taking 
%   $j_\sigma = j_h$ (outside the image).
%   More formally, 
%   
  Now we have to define $j_\sigma$ outside the image of $c^F$.  Whatever we 
  choose, this assignment will define also $j_c$ (by Lemma~\ref{lem:jimplied}). 
  But the composite of $g$ and $c$ will have $j_{cg}$ equal to $j_c$ (since $g^F$
  has empty complement).  So in order to have $c\circ g = h$, we need $j_c=j_h$,
  so we are forced to take  $j_\sigma = j_h$.
  Now we have constructed a factorisation.  For any two possible factorisations
  (which may differ only by the names of the elements in the sets), there is
  a unique comparison, determined by the bijections already involved.  Hence
  the factorisation is essentially unique.  We have established that the two
  classes of morphisms constitute a factorisation system.
\end{proof}

\begin{blanko}{Commutation law.}\label{BMcommlaw}
  We shall need the special case of factorisation where a composable pair of
  arrows in the `wrong order' is composed and then factorised into a pair in the
  `right order'.  This interchange of order is an example of the general notion
  of commutation law.  The axioms for a commutation law will not be listed,
  since they are automatically satisfied for commutation laws given by a
  factorisation system, as is the case here (and in \ref{prop:commlawcommlaw}).
  Given a compression morphism $h:\tau \to \omega$ and a grafting $k: \omega \to
  \rho$,
  $$\xymatrix{
     \omega \ar[r]^k & \rho  \\
     \tau \ar[u]^h\ar@{..>}[r]_g & \sigma \ar@{..>}[u]_c
  }$$
  the commutation law completes the square with $g$ a grafting and $c$ a
  compression.  As in the proof above, the vertices and flags of $\sigma$ are 
  those of $\tau$, with $g_V$ and $g^F$ the identity maps.  Only the involution
  $j_\sigma$ changes: $g$ grafts the `same' tails as $k$ does---this makes sense
  since $h$ is bijective on tails.  The compression $c$ has $c_V=h_V$ and 
  $c^F=h^F$, modulo the bijections that identity vertices and flags of $\omega$ 
  with vertices and flags of $\rho$.  It other words, it compresses the same 
  `subgraphs' as $h$ does. 
\end{blanko}

%%%%%%%%%%%%%%%%%%%%%%%%%%%%%%%%%%%%%%%%%%%%%%%%%%
\section{Joyal--Kock category of graphs and etale morphisms}
%%%%%%%%%%%%%%%%%%%%%%%%%%%%%%%%%%%%%%%%%%%%%%%%%%

The following definition of graph is from
Joyal--Kock~\cite{Joyal-Kock:0908.2675}.  The definition can be seen as an
natural open-ended version of the Serre definition of graph, and it is also in
the spirit of the polynomial tree formalism in \cite{Kock:0807}.  A directed
variant with the same features was given in \cite{Kock:1407.3744}.

\begin{blanko}{Definition of graph.}\label{def}
  A {\em graph} is a diagram of finite sets
  \begin{equation}\label{eq:graph}
    \xymatrix {
    A \ar@(lu,ld)[]_i & H \ar[l]_s \ar[r]^p & V
    }
  \end{equation}
such that $s$ is injective and $i$ is a fixpoint-free involution. 

The set $V$ is the set of {\em vertices}.  The set $H$ is the set of {\em
half-edges} or {\em flags}: these are pairs consisting of a vertex together with
the germ of an emanating arc.  Finally the set $A$ is the set of {\em arcs},
which can be thought of as an edge with a chosen orientation.  The involution
$i$ is thought of as reversing the orientation.  The map $p$ forgets the
emanating arc.  The map $s$ returns the emanating arc in the direction pointing
away from the vertex.  An {\em edge} is by definition an $i$-orbit, hence
consist of two arcs.  An {\em inner edge} is an $i$-orbit both of whose elements
are in the image of $s$.  A {\em port} is by definition an arc in the complement
of the image of $s$.  The set of ports of a graph is called its {\em interface}.
The {\em local interface} of a vertex $v\in V$ is the set of arcs pointing
towards it, formally $i(s(p^{-1}(v)))$.  A few examples and pictures are given
on the next page.
\end{blanko}

At first look, this definition may not seem so different from the definition
in the introduction, and even a bit more redundant. 
Its merit is the natural notion of morphism:

\begin{blanko}{Etale morphisms.}
  An {\em etale morphism} of graphs is a commutative diagram
\begin{equation}\label{eq:etale}
\xymatrix {
A' \ar@(lu,ld)[] \ar[d] & H' \ar[l] \ar[r] \ar[d]  \drpullback& V' \ar[d] \\
A \ar@(lu,ld)[] & H \ar[l] \ar[r] & V ,
}
\end{equation}
in which the right-hand square is a pullback.
The pullback condition says that each vertex must map to a vertex of
the same valence. 
Let $\Gr$ denote the category of graphs and etale morphisms.  
(Note that \cite{Joyal-Kock:0908.2675} considers only connected graphs.)
An {\em open subgraph (inclusion)} is an etale morphism that is injective 
levelwise.  In this note, all subgraphs considered are open, and we 
suppress the adjective `open'.
\end{blanko}

\begin{blanko}{Unit graph and effective graphs.}
  The {\em unit graph}, denoted $\triv$, is the graph with one edge and no 
  vertices.  It is given by
  $$
  \xymatrix {
  2 \ar@(lu,ld)[] & 0 \ar[l]\ar[r] & 0 .
  }
  $$
  This graph does not have any correspondent in the BM category.  In the
  following we will often have to make exception for it, and more generally
  exclude graphs with isolated edges (an edge is isolated if neither of its
  arcs are in the image of $s$).  A graph is {\em effective} if it is
  nonempty and has no isolated edges.  Note that an effective open subgraph 
  is determined by its vertices.
  
  Unit graphs are, however, essential for the colimit features of $\Gr$, treated
  next.
\end{blanko}

\begin{blanko}{Colimits and monoidal structure.}\label{colimits}
  The category $\Gr$ has sums (disjoint union of graphs), computed levelwise
  (i.e.~by disjoint union of the sets involved).  The empty graph is the neutral
  object for the symmetric monoidal structure given by sums.  The category $\Gr$
  also has other useful colimits, namely certain coequalisers and pushouts over
  unit graphs, expressing precisely grafting of graphs at their ports.  These
  colimits are also computed levelwise.  Precisely, a diagram
  $$
  \xymatrix{ \triv \ar[r]<+3.5pt>^-{f} \ar[r]<-3.5pt>_-{g} & \graf{X} 
  \ar@{..>}[r] & \graf{Y} }
  $$
  has a colimit if $f$ and $g$ map to different edges in $\graf{X}$, such that $f$ 
  maps one of the two arcs in $\triv$ to a port and $g$ maps the
  other arc in $\triv$ to another port: the coequaliser $Y$ is
  interpreted as
  the graph obtained by gluing at the two ports.  If $\graf{X}$ is effective
  then also $\graf{Y}$ is effective.
 
  Similarly, pushouts over a unit graph $\graf{X}\stackrel{f}{\leftarrow}
  \triv \stackrel{g}{\to} \graf{X}'$ exist if the corresponding coequaliser
  exists $\xymatrix{ \triv \ar[r]<+3.5pt>^-{f} \ar[r]<-3.5pt>_-{g} &
  \graf{X}+\graf{X}' \ar@{..>}[r] & \graf{Y} }$.  The possibility of handling
  gluing formally in terms of colimits, and the fact that these colimit are
  computed levelwise, is a crucial feature of the category $\Gr$.  
%   (One may note
%   that the BM category of graphs does not have any of these colimits, not even
%   sums: while one may form the disjoint union of two graphs, it is not the
%   categorical sum in the BM category.)
\end{blanko}

\begin{blanko}{Reduced covers.}
  A family of etale morphisms with codomain $\graf{X}$ is called a {\em 
  covering family} of 
  $\graf{X}$ if it
  is jointly surjective on edges and vertices; this defines the etale topology
  on the category $\Gr$.  A covering family is called {\em reduced} when no member of the
  family could be omitted or replaced by a proper subgraph without spoiling the joint
  surjectivity.  In the absence of isolated edges, this means jointly bijective
  on vertices.
  % The empty family covers the empty graph.
  
%   \tiny The word `omitted' is crucial to have, as it excludes the possibility 
%   of an empty member.  (An empty member could not be replaced by a proper 
%   subgraph, because it has none.  But it could be omitted.)  If we don't put
%   the word `omitted', then any number of empty members could occur.  There is
%   no problem with projection morphisms $X+X \to X$ (duplication of members): although
%   the natural way of exhibiting redundancy would be to omit one of the two 
%   members, the redundancy could also be exhibited by replacing one of the two 
%   copies with a smaller graph!
%   
%   \normalsize
  
  Since $\Gr$ has sums, one may regard a covering family as a single etale morphism,
  by taking sum of the domains.  We shall do this consistently.  The single 
  etale morphism obtained from a reduced covering family is called a {\em reduced 
  cover}. 
  % The morphism $\emptyset\to\emptyset$ is a reduced cover (although as a 
  % one-element family it is not reduced).
  The quotient 
  morphism of a coequaliser (or pushout) as in \ref{colimits} is always a reduced 
  cover, and every reduced cover can be described as an iterated coequaliser 
  quotient morphism like this.  In particular (still in the absence of isolated 
  edges), the reduced covers of a given graph $\graf{X}$ form the
  boolean lattice of subsets of the set of inner edges of $\graf{X}$.
\end{blanko}

\begin{prop}\label{prop:JKrc=BMgr}
  Let $\kat{BM}_{\text{\rm gr}}\subset \kat{BM}$ denote the wide subcategory
  with only graftings as morphisms, and let $\Grwie_{\text{\rm rc}}\subset\Gr$
  denote the wide subcategory whose objects are the graphs without isolated
  edges, and whose morphisms are the reduced covers.  There is a canonical
  equivalence
  $$
  \Phi_1 : \kat{BM}_{\text{\rm gr}} \isopil \Grwie_{\text{\rm rc}} .
  $$
\end{prop}

\begin{proof}
  On objects (cf.~also~\cite[15.3]{Batanin-Berger:1305.0086}): given a BM graph
  $(V,F,\partial,j)$, define a JK graph \eqref{eq:graph} with the same $V$,
  with $H:=F$ and $A:= F+T$ (the disjoint union of the set of all flags and the
  set of all tails), with $s$ the sum inclusion and $p:=\partial$.  The
  fixpoint-free involution $i:A\to A$ is given by $j_\tau$ on the fixpoint-free
  part of $F$, and interchanges each fixed flag in $F$ with its corresponding copy in
  $T$.  The part $T$ plays the role of ports.  It is clear that there are no
  isolated edges.
  Conversely, given a JK graph without isolated edges, take the same $V$,
  put $F:=H$, and define an involution $j: F \to F$ by $j(f)=s^{-1}isf$ if
  $f$ is part of an inner edge, and $j(f)=f$ otherwise.
    
  On morphisms: given a grafting $h:\tau\to \sigma$ between BM graphs, then
  since $h^F: F_\sigma\to F_\tau$ is bijective, we can use its inverse to define
  the component on half-edges between the corresponding JK graphs.  The required
  map on arcs $A_\tau \to A_\sigma$ has to be a map $F_\tau + T_\tau
  \to F_\sigma + T_\sigma$.  For the left-hand square of \eqref{eq:etale}
  to commute, we are forced
  to take $f \mapsto (h^F)^{-1}(f)$ for $f\in F_\tau$.  For the map to be
  compatible with the fixpoint-free involutions $i$, we are also forced to take
  $f \mapsto j_\sigma(h^F)^{-1}(f)$ for $f\in T_\tau$.
  % This requires thought.
  Conversely, given a reduced cover between JK graphs without isolated edges,
  since it is bijective on vertices, it is also bijective on half-edges (by the
  pullback condition), so we can obtain a grafting between the corresponding BM
  graphs by taking the inverse on flags, and simply forgetting the component on
  arcs.
\end{proof}

%%%%%%%%%%%%%%%%%%%%%%%%%%%%%%%%%%%%%%%%%%%%%%%%%%
\section{Graphical species, monad, and Kleisli category}
%%%%%%%%%%%%%%%%%%%%%%%%%%%%%%%%%%%%%%%%%%%%%%%%%%

\begin{blanko}{Elementary graphs.}
  An {\em elementary graph} is a connected graph without inner edges.
  (Connected means: nonempty and not the sum of smaller graphs.)
%   An elementary graph is thus either a unit graph or a corolla.
  Here are the first few elementary graphs:
  \begin{center}\begin{texdraw}
  \lvec (0 20)
  \move (20 10) \onedot
  \move (40 10) \onedot \lvec (40 21)
  \move (60 10) \onedot \lvec (65 20) \move (60 10) \lvec (65 0)
  \move (80 10) \onedot \lvec (75 20) \move (80 10) \lvec (75 0) \move (80 10) 
  \lvec (90 10)
  \move (105 10) \onedot \move (97 18) \lvec (113 2)
  \move (97 2) \lvec (113 18)
  \end{texdraw}\end{center}
  The first one is the unit graph already mentioned;
  the remaining ones are the {\em corollas} $\graf{n}$, given by 
  $$
  \xymatrix  @!0 @C=11pt  {
  \ar@(lu,ld)[] \phantom{n}&n+n'&&&&& \ar[lllll] n' \ar[rrrr]&&&& 1,
  }
  $$
  where $n$ and $n'$ are finite sets with a bijection $n \isopil n'$ defining 
  the involution.  (The set of ports of $\graf{n}$ is thus $n$.)
  Let $\elGr$ denote the full subcategory of
  $\Gr$ consisting of the elementary graphs. 
% We have:
% \begin{align*}
%   \Hom(\triv,\triv)\ = & \ 2\\
%   \Hom(m, n) \ = & \ \begin{cases} n! & \text{ if }m=n\\ 0 & \text{ if } 
%   m\neq n\end{cases} \\
%   \Hom(\triv , n) \  = & \ 2n \\
%   \Hom(n, \triv) \ = & \ 0. 
% \end{align*}
% 
\end{blanko}

It is easy to check that every graph $G$ is canonically a colimit in $\Gr$
of its elementary subgraphs.  The indexing category is the category of elements 
$\el(\graf{G}) := \elGr\comma \graf{G}$.
 The canonical colimit decomposition of a graph is also a 
canonical cover, and it follows readily that there is an 
equivalence of categories between presheaves on $\elGr$ and sheaves on $\Gr$, 
for the etale topology:
$$
\PrSh(\elGr) \isopil \Sh(\Gr) .
$$
% (More formally, the full inclusion of categories
% $\elGr \into \Gr$ induces an essential geometric embedding of presheaf toposes
% $\PrSh(\elGr) \to \PrSh(\Gr)$ and it is well-known (see 
% \cite{MacLane-Moerdijk}, Ch.VII) that every such induces a unique topology
% on $\Gr$ giving the above equivalence.)

\begin{blanko}{Graphical species.}
  A presheaf $F:\elGr\op\to\Set$ is called a {\em graphical species}
  \cite{Joyal-Kock:0908.2675}; its value on $\graf{n}$ is denoted $F[\graf{n}]$.
  Explicitly, a graphical species is given by an involutive set $C = F[\triv]$
  (of `colours'), and for each $n\in \N$ a set $F[\graf{n}]$ (of `operations')
  with $2n$ projections to $C$, permuted by a $\mathfrak S_n$-action on
  $F[\graf{n}]$ and by the involution on $C$.

  Each graphical species $F$ defines a notion of {\em $F$-graph}: these are 
  graphs decorated on edges by the colours of $F$ and on vertices by the 
  operations of $F$.  More formally, the category of $F$-graphs is the comma
  category $\Gr\comma F$.

  The equivalence $\PrSh(\elGr) \simeq \Sh(\Gr)$ means that every graphical
  species can be evaluated not only on elementary graphs but on all graphs: if
  $F$ is a graphical species and $\graf{G}$ is a graph, then
  $$
  F[\graf{G}] = \lim_{\graf{E}\in\el(\graf{G})} F[\graf{E}]   ,
  $$
  where $\graf{E}$ runs over the category of elements of $\graf{G}$, i.e.~all
  the elementary subgraphs of $\graf{G}$ and the way they are glued together to
  give $\graf{G}$. 
%   Note that for every $F\in \PrSh(\elGr)$ we have $F[\emptyset]=1$.
\end{blanko}

\begin{blanko}{The idea of generating more graph morphisms by a monad.}
  The etale morphisms, and in particular the injective ones, serve as the geometric
  fabric supporting further algebraic features, essentially graph substitution,
  which will now be generated by a monad defined in terms of colimits.
  The new morphisms generated allow for refining a graph by substituting a foreign
  graph into a vertex. 
  For the notions of monad and Kleisli category we refer to Mac
  Lane~\cite{MacLane:categories}, Ch.VI.

  Generating new morphisms by means of a monad is an important construction in the
  combinatorics of higher structures \cite{Weber:TAC13}.  As important examples,
  the simplex category $\Delta$ arises in this way from
  the free-category monad, Joyal's cell category $\Theta$ from the
  free-strict-$\omega$-category monad \cite{Berger:Adv}, the Moerdijk--Weiss
  category $\Omega$ of trees from the free-operad monad \cite{Kock:0807}.

  The following description of the monad and its restricted Kleisli category is
  essentially that of Joyal--Kock~\cite{Joyal-Kock:0908.2675}, but with 
  modified conditions imposed on $\graf{n}$-graphs.  These conditions control
  what kind of
  graphs are allowed to be substituted into a vertex, and it is modified in
  order to accommodate the Borisov--Manin mergers.  The adjustments are
  explained in \ref{adjusts} below.
  Monads of this flavour---defined as a colimit of limits---abound in the
  literature on operads and related structures.  The general theory of such
  monads is perhaps best cast in a polynomial formalism, for which the 
  current state-of-the-art is Batanin--Berger~\cite{Batanin-Berger:1305.0086},
  but going in this direction is beyond the scope of this note.
\end{blanko}

\begin{blanko}{$\graf{n}$-graphs.}
  Let $n$ be a finite set.  An {\em $\graf{n}$-graph} is a graph $\graf{G}$
  whose set of ports is $n$ (or more precisely, equipped with a bijection with
  the ports of $\graf{n}$).  A morphism of $\graf{n}$-graphs is an isomorphism
  leaving the set of ports fixed (or more precisely, compatible with the
  specified bijections).  We shall need only {\em effective $\graf{n}$-graphs}
  and their isomorphisms, forming a groupoid denoted
  $\graf{n}\kat{-Gr}_{\iso}$.
\end{blanko}

\begin{blanko}{Monad \cite{Joyal-Kock:0908.2675}.}
  Define a monad by the assignment
  \begin{align*}
     \PrSh(\elGr) \ \longrightarrow & \ \PrSh(\elGr) \\  
     F \ \longmapsto & \ \ov F ,
  \end{align*}
  where $\ov F$ is the graphical species given by $\ov F[\triv] :=  
  F[\triv]$ and
  $$
  \ov F[\graf{n}] \ :=  \ \colim_{\graf{G} \in \graf{n}\kat{-Gr}_{\iso}} 
  F[\graf{G}] \ = \ \sum_{\graf{G}\in \pi_0(\graf{n}\kat{-Gr}_{\iso})}  
  \frac{F[\graf{G}]}{\Aut_{\graf{n}}(\graf{G})} \ 
  = \ \pi_0 \big( \graf{n}\kat{-Gr}_{\iso} \comma F \big) .
  $$

  The monad multiplication is described as follows.
  $\ov F[\graf{n}]$ is the set of isomorphism classes of effective
  $\graf{n}$-$F$-graphs: it is the set of ways to decorate effective
  $\graf{n}$-graphs by the graphical species $F$.  Now $\ov { \ov F} [\graf{n}]$
  is the set of effective $\graf{n}$-graphs decorated by effective $F$-graphs:
  this means that each vertex is decorated by an $F$-graph with matching
  interface.  We can draw each vertex as a circle with the decorating $F$-graph
  inside, and the monad multiplication then consists in erasing these circles,
  turning a graph with vertices decorated by $F$-graphs into a single $F$-graph.
  More formally, we can use the $\graf{n}$-graph as indexing a diagram of
  $F$-graphs, and then take the colimit in $\Gr$.  In detail, the
  groupoid $\Gr_{\iso}\comma \ov F$ has as objects pairs
  $(\graf{R},\phi)$ where $\graf{R}$ is a graph, and
  $\phi: \Hom( -, \graf{R}) \to \ov F$ is a
  natural transformation.  Equivalently we can regard $\phi$ as a functor
  $\el(\graf{R}) \to \elGr\comma \ov F$.  Now there is also a canonical functor
  $\elGr\comma \ov F \to \Gr\comma F$, which takes unit graphs to unit graphs,
  and takes a corolla decorated by an $F$-graph to that same $F$-graph.
  % Warning: in reality there is a $\pi_0$ involved: an object in 
  % $\elGr \comma \ov F$ is given by a $(m,n)$ and an element in $\ov F[m,n] =
  % \pi_0 ((m,n)\kat{-Gr}_{\iso} \comma F)$.  The claimed canonical functor 
  % involves the inclusion $\pi_0 ((m,n)\kat{-Gr}_{\iso} \comma F) \longrightarrow
  % \Gr\comma F$ which to an isoclass of an $(m,n)$-$F$-graph associates
  % any point of that connected component.  So it is not more canonical than
  % up to non-unique isomorphism.  This is good enough to compute a colimit,
  % though.
  We take the colimit of the composite functor $\el(\graf{R}) \to \elGr\comma
  \ov F \to \Gr\comma F$ to obtain a single $F$-graph.  The whole construction
  defines a functor
  $$
  \Gr_{\iso}\comma \ov F \to \Gr_{\iso}\comma F ,
  $$
  whose fibrewise $\pi_0$ defines the monad multiplication.
  The unit for the monad interprets an $F$-corolla as an $F$-graph.
  (More details and a careful proof of associativity can be found in
  \cite{Kock:1407.3744}, in the case of directed graphs.)
\end{blanko}

\begin{blanko}{`Compact coloured props'.}
  Let $\kat{CCP}$ denote the category of algebras for the monad $F\mapsto \ov
  F$,  tentatively called {\em compact coloured props}
  (although rightly the word `effective' ought to be part of the name).
  Hence a compact coloured
  prop is a graphical species $F:\elGr\op\to\Set$ equipped with a
  structure map $\ov F \to F$: it amounts to a rule which for any
  effective $\graf{n}$-graph $\graf{G}\in \graf{n}\kat{-Gr}_{\iso}$ gives a map
  $F[\graf{G}] \to F[\graf{n}]$, i.e.~a way of constructing a single operation
  from a whole graph of them.  This rule satisfies an associativity condition
  (cf.~\cite{MacLane:categories}, Ch.VI),
  amounting roughly to independence of
  the different ways of breaking the computation into steps.
\end{blanko}

\begin{blanko}{Choices involved in the definition of the CCP monad.}\label{adjusts}
  By adjusting the conditions imposed on $\graf{n}$-graphs, one can get
  different monads.  The original choice of Joyal--Kock~\cite{Kock:1407.3744} is
  to allow only connected $\graf{n}$-graphs.  This excludes the empty graph, but
  allows the unit graph, and leads to a notion of compact symmetric
  multicategory, essentially a coloured version of the modular operads of
  Getzler and Kapranov~\cite{Getzler-Kapranov:9408}.
  
  For the present purposes the connectedness requirement is given up, in order
  to be general enough to model also prop-like structures, and not just
  operad-like structures.  Specifically it is the availability of non-connected
  graphs in $\graf{n}\kat{-Gr}_{\iso}$ that leads to morphisms corresponding to
  the mergers of Borisov--Manin.  On the other hand we choose not to allow the
  empty graph in $\graf{0}\kat{-Gr}_{\iso}$.  This choice is merely taken in
  order to agree with Borisov--Manin.  (Actually the empty graph is a perfectly
  valid $\graf{0}$-graph, and it could be included in $\graf{0}\kat{-Gr}_{\iso}$
  without problems.  It would actually make sense to include it, so as to allow
  an isolated vertex to be refined into the empty graph, in turn a substitution
  relevant for the graphical modelling of props, where it accounts for the
  neutral element in the monoid of $(0,0)$-operations, not guaranteed in the
  formalisms of \cite{Borisov-Manin:0609748} or \cite{Kaufmann-Ward:1312.1269}.
  See also Remark~10.5 of Batanin--Berger~\cite{Batanin-Berger:1305.0086} for
  related discussion.)
% 
%   Incorporating this morphism in the BM category would require allowing 
%   non-surjective vertex maps $h_V: V_\tau \to V_\sigma$, 
%   allowed to miss isolated vertices in the codomain.)
% 
%   
%   This would correspond in their language to allow
%   contracting an empty subgraph to a vertex.  This is needed in order to
%   properly handle the case of (graphical) props.  This is a rather subtle
%   issue\ldots Excluding the empty graph from $\graf{0}\kat{-Gr}_{\iso}$ is
%   not necessary for the sake of the theory; this choice is made only to agree
%   with Borisov--Manin.)
%   
  Finally we choose to exclude the unit graph $\triv$ from
  $\graf{2}\kat{-Gr}_{\iso}$, and with it all graphs with isolated edges.  A
  superficial reason is to arrive precisely at the BM category, where unit
  graphs are not allowed.  But a deeper reason is that
  Proposition~\ref{prop:pushout} below would not be true if we allowed $\triv$
  as an object in $\graf{2}\kat{-Gr}_{\iso}$.
\end{blanko}

\begin{blanko}{Kleisli category.}
  Recall that the Kleisli category (see Mac
  Lane~\cite[\S\,VI.5]{MacLane:categories}) is the full subcategory of the
  category of algebras spanned by the free algebras.  By adjunction, this can
  also be described as the category whose objects are graphical species, and
  where a morphism $F \to Y$ is a morphism of graphical species $F \to \ov Y$.
  We are interested in the full subcategory $\widetilde{\Gr}$ of the Kleisli
  category given by restricting to graphs, via the fully faithful embedding
  $\Gr\to \PrSh(\elGr)$.  In \cite{Joyal-Kock:0908.2675} this meant restricting
  to connected graphs (including the unit graph).  In the present modified
  version we allow non-connected graphs, but exclude graphs with isolated edges,
  such as most importantly the unit graph.  We do allow the empty graph as
  object in the restricted Kleisli category $\widetilde \Gr$ (although we have
  excluded it from $\graf{0}\kat{-Gr}_{\iso}$).
    
  In intuitive terms, a morphism in this restricted Kleisli category $\widetilde
  \Gr$ is a map that sends edges to edges and sends vertices to effective
  subgraphs with the same interface, and respects the flag incidences.  More
  formally, a Kleisli morphism from graph $\graf{R}$ to graph $\graf{Y}$ is a
  morphism of graphical species $\graf{R}\to \ov{\graf{Y}}$.  This can also be
  described as a functor $\el(\graf{R}) \to \Gr\comma \graf{Y}$, with the
  property that each edge element of $\el(\graf{R})$ is sent to a unit graph
  (over $\graf{Y}$) and each $\graf{n}$-vertex element $x$ is sent to an
  effective $\graf{n}$-graph $\graf{S}_x$ (over $\graf{Y}$).  (The vertex
  condition involves a specific bijection between the interface of the
  $\graf{R}$-corolla at $x$ and the interface of the graph $\graf{S}_x$.)  The
  colimit of this diagram exists in $\Gr\comma \graf{Y}$, and is denoted $\graf{S}$: it is the graph
  obtained by gluing together all the graphs $\graf{S}_x$ according to the same
  `recipe' $\el(\graf{R})$ that serves to assemble $\graf{R}$ from its elements.
  The graph $\graf{S}$ is etale over $\graf{Y}$ by construction. 
  (The formal
  colimit description of the various notions involved has been detailed in
  \cite{Kock:1407.3744}, in the case of directed graphs.)
  
  There is a subtlety here: for all colimit computations we work in $\Gr$ where
  the unit graph is a crucial element.  However, we only glue effective graphs,
  and the result is always effective again, so that both input and output to the
  computation are objects in $\widetilde\Gr$.
\end{blanko}

The following lemma is easy, but important.
\begin{lemma}\label{lem:epi}
  Every cover morphism $\graf{R}\to \graf{R}'$ 
  is an epimorphism in the Kleisli category.  That is, the extension
    $$\xymatrix{
     \el(\graf{R}) \ar[r]\ar[rd] & \el(\graf{R}') \ar@{..>}[d] \\
       & \Gr\comma \graf{Y} 
  }$$
  is unique if it exists.
\end{lemma}
\begin{proof}
  This follows since the functor $\el(\graf{R}) \to 
  \el(\graf{R}')$
  is surjective on objects, and in particular is an epi-functor.
\end{proof}

\begin{blanko}{Generic morphisms and generic factorisations.}
  We say that a Kleisli morphism $\graf{R} \to \graf{Y}$ as above is {\em generic} if
  the colimit graph $\graf{S}$ is terminal in $\Gr\comma \graf{Y}$, or,
  equivalently, if the etale morphism $\graf{S} \to \graf{Y}$ is invertible.  We
  shall see in a moment that generic morphism have a natural interpretation
  as refining the vertices into effective graphs with matching interfaces.
  
  It follows from the previous descriptions that every Kleisli morphism $\graf{R} \to
  \graf{Y}$ has a canonical factorisation
  into a generic morphism $\graf{R}\to\graf{S}$ followed by a free morphism
  $\graf{S} \to \graf{Y}$.  `Free' just means that it is the image of a etale
  morphism under the monad.  Such {\em generic factorisations} can be deduced
  formally from properties of the monad~\cite{Weber:TAC13}.  (In fact, these
  abstract properties of the monad allows one to characterise the graphs among
  all graphical species as those objects that can arise as the middle object of
  a generic factorisation of a Kleisli morphism from a representable (i.e.~an
  elementary graph.))
%   The factorisation generic-free does not quite form an
%   orthogonal factorisation system, due to the presence of deck transformations
%   of etale covers (see \cite{Kock:1407.3744} for discussion of this issue).  (In
%   many cases it {\em is} actually unique, for example when the etale morphism is
%   injective on vertices.)
\end{blanko}

% NEED TO NOTE THAT in general the comparison isomorphism comparing two 
% factorisations is not unique.  But is unique in many 
% cases, such as when the free morphism is a reduced cover (since such etale 
% morphism
% cannot have deck transformations\ldots)  Reduced covers are not quite mono
% in the category of etale morphisms, because they are not injective on edges.
% However, they are mono with respect to isomorphisms: if two parallel isos
% into $\graf{S} \to \graf{S}'$ have equal composites then they are already 
% equal!

%%%%%%%%%%%%%%%%%%%%%%%%%%%%%%%%%%%%%%%%%%%%%%%%%%
\section{Generic morphisms (refinements) versus compressions}
%%%%%%%%%%%%%%%%%%%%%%%%%%%%%%%%%%%%%%%%%%%%%%%%%%

\begin{blanko}{Generic morphisms as graph refinements.}
  From the general description of the Kleisli morphisms, we see that generic
  morphisms out of a graph $\graf{R}$ are given by collections of
  effective graphs $\graf{S}_x$ indexed by the vertices of $\graf{R}$, which are
  then glued together in the same way as the corollas in $\graf{R}$
  are glued together to give $\graf{R}$. Furthermore, each $\graf{S}_x$
  has the same interface as the corolla at $x$.  Geometrically,
  what the generic morphism does is therefore to {\em refine} each vertex in $\graf{R}$
  into an effective graph with the same interface.
\end{blanko}

\begin{blanko}{Duality between generic morphisms and reduced covering families.}
  Each generic morphism with codomain $\graf{S}$ gives rise to a reduced 
  covering family
  of $\graf{S}$, namely $\big(\graf{S}_x \to \graf{S} \mid x\in V
  \big)$ corresponding to the quotient morphism $\sum_x \graf{S}_x \to \graf{S}$
  in the colimit description (it is reduced since the graphs $\graf{S}_x$ are
  effective).  Conversely, every reduced covering family $\big(\graf{S}_x \to \graf{S}
  \mid x\in V \big)$
%   (really given as a jointly
%   surjective family of morphisms, not just given as a single etale morphism
%   $\sum_x \graf{S}_x \to \graf{S}$),
  is the quotient morphism of a canonical coequaliser diagram of graphs of form
  $$
  \xymatrix{ \sum_e \triv \phantom{x} \ar[r]<+3.5pt> \ar[r]<-3.5pt> & \phantom{!}\sum_x \graf{S}_x }
  $$
  where $e$ runs over the set of edges of $\graf{S}$ hit twice by the cover.
  Since the cover is reduced, no edge is hit more than twice, and all the
  $\graf{S}_x$ are effective (but not necessarily connected).
  The same shape of coequaliser diagram can be formed
  with a corolla $\graf{n}_x$ in place of each $\graf{S}_x$:
  $$
  \xymatrix{ \sum_e \triv \phantom{x} \ar[r]<+3.5pt> \ar[r]<-3.5pt> & 
  \phantom{!} \sum_x \graf{n}_x   .}
  $$
  The colimit of this diagram is a graph $\graf{R}$ with a vertex for each $x\in
  V$.
  The assignment $x\mapsto \graf{S}_x$ now defines a generic morphism $\graf{R}
  \to \graf{S}$.  It is readily seen that these two constructions are
  essentially inverse to each other.  This `duality' between reduced covering
  families and generic morphisms was first observed by Berger~\cite{Berger:Adv}
  in the cases of $\Delta$ and $\Theta$; see \cite[Remark 2.3.2]{Kock:0807} for
  the case of trees.
\end{blanko}

\begin{blanko}{Diagrammatic description of generic morphisms.}
  It is a key feature of the JK formalism that etale morphisms are given by
  straightforward diagrams \eqref{eq:etale}.  A similar diagrammatic description
  is possible for refinement morphisms.  Namely, a refinement morphism $\graf{R}
  \to \graf{S}$ is specified unambiguously by a diagram
  \begin{equation}\label{eq:Sub}
  \xymatrix {
  A_{\graf{R}} \ar@(lu,ld)[] \ar[d] & H_{\graf{R}} \ar[l] \ar[r] \ar[d]  \drpullback& V_{\graf{R}} \ar[d] \\
  A_{\graf{S}} \ar@(lu,ld)[] & \Sub'(\graf{S}) \ar[l] \ar[r] & \Sub(\graf{S}) ,
  }
  \end{equation}
  in which the right-hand square is a pullback, and required to be bijective on 
  ports.  The bottom diagram is not a graph, only a graphical species.
  Here
  $\Sub(\graf{S})$ is the set of effective subgraphs of $\graf{S}$, and
  $\Sub'(\graf{S})$ is the set of effective subgraphs of $\graf{S}$ with a specified 
  outer flag.  The {\em outer flags} of an (effective) subgraph are $i$ of its ports.  
  Note that applying the involution ensures that the notion agrees with the flags of
  a vertex considered as a one-vertex subgraph, and let us define {\em ports} of
  a diagram like the bottom row as those arcs that are not $s$ of a
  flag of any effective subgraph.  It is clear that the ports of this bottom row
  are just the ports of $\graf{S}$.
  
  Indeed, given a generic morphism, the assignment $x\mapsto \graf{S}_x$ defines
  the right-most vertical map.  The middle vertical map is canonically 
  defined since each
  flag belongs to a vertex $x$, and the data in a generic morphism includes
  bijections between the flags of $x$ and the flags of $\graf{S}_x$.
  Commutativity of the left-hand square and with the involutions expresses
  precisely that the gluing data is the same for $\graf{S}$ from the
  $\graf{S}_x$ as for $\graf{R}$ from its vertex subgraphs.  From this it is
  also clear that the left-hand map induces a bijection on ports.  Conversely,
  given such a diagram, the assignment $x\mapsto \graf{S}_x$ is already given by
  the right-hand map, and it has the correct interfaces because the second square
  is a pullback.  It is routine to check that the two constructions are inverse
  to each other.
\end{blanko}

\begin{prop}\label{prop:JKgen=BMcomp}
  Let $\kat{BM}_{\text{\rm comp}} \subset \kat{BM}$ denote the wide subcategory with
  only the compression morphisms, and let $\widetilde\Gr_{\text{\rm gen}}\subset 
  \widetilde\Gr$ denote the wide subcategory with only generic morphisms.
  There is a canonical equivalence $$\Phi_2: \kat{BM}_{\text{\rm comp}} \isopil 
  \widetilde\Gr_{\text{\rm gen}}\op.$$
\end{prop}
\begin{proof}
  We define functors in both directions.
  
  Given a generic morphisms $\graf{R} \to \graf{S}$ in terms of
  a diagram \eqref{eq:Sub},  let $\rho$ and $\sigma$ be the corresponding BM
  graphs, as in the object-part of the proof of \ref{prop:JKrc=BMgr}.
  We need to construct a compression morphism $h:\sigma\to\rho$.  For
  each vertex $x\in V_{\graf{R}}$, we have an effective subgraph $\graf{S}_x \subset
  \graf{S}$,
  % with at least one vertex,
  and we have a reduced cover $\sum_x \graf{S}_x \to \graf{S}$.  Since it is
  bijective on vertices, every vertex $v$ of $\graf{S}$ belongs to a unique
  subgraph $\graf{S}_x$.  We define the map $h_V: V_\sigma \to V_\rho$ by
  sending $v$ to $x$.  It is surjective as required since each subgraph
  $\graf{S}_x$ has at least one vertex by the effectivity assumption.  Now we
  define $h^F: F_\rho \to F_\sigma$ to be the composite $H_{\graf{R}} \to
  \Sub'(\graf{S}) \to H_{\graf{S}}$.  This map is injective as a consequence of
  the fact that the graphs $\graf{S}_x$ are effective.  Indeed, each flag in
  $\graf{R}$ is {\em either} $i$ of a port, in which case its image is too, and
  generic morphisms are bijective on ports, {\em or} it is part of an inner edge
  of $\graf{R}$, which means that it appears once for gluing $\graf{R}$ from its
  elementary graphs, and hence also appears once for gluing together $\graf{S}$
  from the $\graf{S}_x$.  Since none of the $\graf{S}_x$ have isolated edges, no
  step in the gluing procedure can collapse flags.  Since a generic morphism
  induces a bijection on ports, it is also clear that $h^F$ is a bijection on
  tails, and it sends edges to edges (hence verifying (iv)), just by
  commutativity of the left hand square and the involutions, and axiom (ii)
  follows also from this.  We need to verify axiom (iii), that endpoints of
  contracted edges of $h$ have equal image under $h_V$.  Contracted edges in
  $\sigma$ is the same thing as inner edges of $\graf{S}$ not in the image of
  $\graf{R}$.  Let $v,v'\in V_\sigma$ be the endpoints of a contracted edge of
  $\sigma$.  That the edge is not the image of an inner edge in $\graf{R}$ means
  that $v$ and $v'$ belong to the same subgraph $\graf{S}_x$, and hence they
  have the same image $x$ under $h_V$ as required.  Finally, it is clear that $h$ is a
  compression (so axiom (ii') is void).

  Conversely, given a compression morphism $h: \sigma\to\rho$ in the BM 
  category, we must construct a diagram like \eqref{eq:Sub}.  For each $x\in 
  V_\rho$, let $\graf{S}_x$ denote the open subgraph of $\graf{S}$ spanned by
  the set of vertices $h_V^{-1}(x)$; it is effective since $h_V$ is surjective.
  The map $h^F: F_\rho \to F_\sigma$ already gives us $H_{\graf{R}} \to 
  H_{\graf{S}}$.  For each vertex $x\in V_\rho$, the map $h^F$ restricts to a
  bijection $\partial^{-1}(x) \isopil h^F(F_\rho) \cap \partial^{-1} h_V^{-1}(x)$ (by 
  Lemma~\ref{lem:boundarypres}),
  which tells us on one hand that each flag in $H_{\graf{R}}$ incident to $x$
  is sent to a flag of the subgraph $\graf{S}_x$, so as to give altogether 
  a well-defined map $H_{\graf{R}} \to \Sub'(\graf{S})$, and on the other hand
  that the right-hand square commutes and is a pullback. Finally we take the map
  $A_{\graf{R}} \to A_{\graf{S}}$ to be
  $$
  h^F + h^F|_{T_\rho} : F_\rho +T_\rho  \to F_\sigma+ T_\sigma,
  $$
  well-defined since $h$ is a compression and hence $h^F$ preserves tails.  The
  left-hand square now commutes by construction.  Since $h^F$ also preserves
  edges (by Lemma~\ref{lem:hj=jh}), this map commutes with the involutions.
  Since $h^F$ is bijective on tails, the new map induces a bijection on ports, as
  required.
  
  It is straightforward to check that the two constructions are inverse to
  each other.
\end{proof}

%%%%%%%%%%%%%%%%%%%%%%%%%%%%%%%%%%%%%%%%%%%%%%%%%%
\section{Borisov--Manin  morphisms as cospans in the Kleisli category}
%%%%%%%%%%%%%%%%%%%%%%%%%%%%%%%%%%%%%%%%%%%%%%%%%%

\begin{prop}\label{prop:pushout}
  In the restricted Kleisli category $\widetilde \Gr$, reduced covers and generic morphisms
  admit pushouts along each other, and the resulting two new morphisms are
  again a reduced cover and a generic morphism:
    $$\xymatrix @!@+6pt {
     \graf{R} \ar[r]^{\text{\rm red.cov.}}\ar[d]_{\text{\rm gen.}} & \graf{R}' 
     \ar@{..>}[d]^{\text{\rm gen.}} \\
     \graf{S} \ar@{..>}[r]_{\text{\rm red.cov.}} & \ulpullback \graf{S}'
  }$$
\end{prop}

Intuitively, $\graf{S}'$ is obtained by refining $\graf{R}'$ by the exact same
prescription as $\graf{S}$ refines $\graf{R}$.  This makes sense since
$\graf{R}'$ and $\graf{R}$ have the `same vertices'.  Equivalently, $\graf{S}'$
is obtained by gluing some ports of $\graf{S}$ in the same way as $\graf{R}'$ is
obtained from $\graf{R}$.  This makes sense since $\graf{S}$ and $\graf{R}$ have
the `same ports'.  The proof formalises this in terms of colimits:

\begin{proof}
  Since a reduced cover can be realised as a sequence of quotient morphisms of 
  simple coequaliser diagrams, it is enough to treat pushout along such a morphism, 
  so we assume that the reduced cover $\graf{R}\to \graf{R}'$ is the quotient
  $\triv' \rightrightarrows \graf{R} \to \graf{R}'$ expressing the gluing of
  two ports of $\graf{R}$ to form an inner edge of $\graf{R}'$.
  The generic morphism $\graf{R}\to \graf{S}$ is given by a functor 
  $\el(\graf{R}) \to \Gr$, which in turn unpacks to a coequaliser diagram
  $$
  \xymatrix{ 
  \sum_e \triv_e \phantom{x} \ar[r]<+3.5pt> \ar[r]<-3.5pt> & 
  \phantom{!}\sum_x \graf{S}_x  \ar[r] & \graf{S}
  }
  $$
  where $e$ runs over the inner edges of $\graf{R}$ and $x$ runs over the 
  vertices of $\graf{R}$, and the parallel arrows express the incidences of 
  inner edges and vertices in $\graf{R}$. 
  The extra gluing producing $\graf{R}'$ from $\graf{R}$
  can now be added to this diagram as an extra summand in the edge sum.
  This is possible: since that extra gluing connects two distinct port-edges 
  of $\graf{R}$ it will also connect two distinct port-edges of $\graf{S}$,
  because a 
  generic morphism is injective on edges and preserves ports.  This gives another
  coequaliser diagram, pictured as the bottom row:
  $$
  \xymatrix{ 
  \phantom{\graf{I}' +} \sum_e \triv_e \phantom{x} 
  \ar[d] \ar[r]<+3.5pt> \ar[r]<-3.5pt> & 
  \phantom{!}\sum_x \graf{S}_x  \ar[d]^= \ar[r] & \graf{S}  \ar@{..>}[d]\\
  \graf{I}' + \sum_e \triv_e \phantom{x} \ar[r]<+3.5pt> \ar[r]<-3.5pt> & 
  \phantom{!}\sum_x \graf{S}_x  \ar[r] & \graf{S}' 
  }
  $$
  expressing a generic morphism $\graf{R}'\to\graf{S}'$.
  The universal property of $\graf{S}$ as a coequaliser now induces the dotted 
  arrow,
  which is the reduced cover given precisely by gluing those two ports,
  i.e.~is the quotient morphism
  $\triv'   \rightrightarrows \graf{S} \to \graf{S}'$.
  We now have the promised square of generic and etale morphisms.
  It commutes because the Kleisli morphism $\graf{R}\to \graf{R}'\to 
  \graf{S}'$ is given by $\el(\graf{R}) \to \Gr\comma \graf{S}'$ with colimit
  $\graf{S}$ (over $\graf{S}'$) by construction.  This shows that
  the generic-free factorisation is $\graf{R} \to \graf{S} \to \graf{S}'$ as
  required.  
  Finally we check that this square is a pullback in the Kleisli 
  category.  Given a commutative square of Kleisli morphisms
  $$\xymatrix{
     \graf{R} \ar[r]\ar[d] & \graf{R}' \ar[d] \\
     \graf{S} \ar[r] & \graf{T}
  }$$
  we need to show there is a unique Kleisli morphism $\graf{S}' \to \graf{T}$ making 
  the triangles commute.  There can be at most one
  extension
  $$\xymatrix{
     \graf{S} \ar[r]\ar[rd] & \graf{S}' \ar@{..>}[d] \\
       & \graf{T}  ,
  }$$
  because $\graf{S}\to\graf{S}'$ is an epimorphism by Lemma~\ref{lem:epi}.
  The graphs $\graf{S}$ and $\graf{S}'$ differ only in 
  that two ports have been glued, so the extension exists if and only if the morphism to
  $\graf{T}$ glues the two ports.  But the two ports come
  from $\graf{R}$, and via $\graf{R}\to\graf{R}'\to\graf{T}$ they are certainly
  glued, so $\graf{S}'\to\graf{T}$ is unique as required.
\end{proof}

% \tiny
% 
% If we allow vertex-injective morphisms, then in the cospan category there are two
% different morphisms from the empty graph to the $0$-corolla.  Namely
% $$
% \xymatrix{
%    \emptyset \ar[r]^= & \emptyset &\ar[l]_{\text{refine}} \bullet
% }
% \qquad
% \xymatrix{
%    \emptyset \ar[r]^{\text{incl.}} & \bullet &\ar[l]_= \bullet
% }
% $$
% 
% \normalsize

\begin{blanko}{Remarks.}
  (1) The proposition is analogous to Lemma~2.4.7 in \cite{Kock:1407.3744}, which
  concerns open inclusions and generic morphisms in the category of acyclic
  directed graphs.
%   More proof details can be found there. 
  The statement is true more generally for etale morphisms that are injective on
  vertices, but not for all etale morphisms.  It would not be true either if we
  permitted the unit graph to be substituted into a $2$-valent vertex.  Formally
  this would break an injectivity condition.  Essentially, the problem is the
  following.  The substitution of a unit graph into the vertex of a graph
  consisting of only one vertex and one loop edge ought to yield a loop without
  vertices.  But this is not a valid graph in $\Gr$ (and is not included in the
  Borisov--Manin definition either).  This may be seen as a shortcoming of both
  definitions.  References \cite{Markl-Merkulov-Shadrin}, \cite{Yau-Johnson},
  \cite{Hackney-Robertson-Yau}, \cite{Batanin-Berger:1305.0086} address the
  issue with the vertex-less loop in various ways.

  (2) Generic-free pushouts in the simplex category
  $\Delta$ play an essential role in the theory of decomposition
  spaces~\cite{Galvez-Kock-Tonks:1512.07573} (also called unital $2$-Segal
  spaces~\cite{Dyckerhoff-Kapranov:1212.3563}): a decomposition space is
  a simplicial space sending generic-free pushouts to pullbacks,
  precisely the condition ensuring associativity of incidence algebras.
  The proposition suggests that it might be possible to develop analogous theory for
  compact coloured props.
\end{blanko}

\begin{blanko}{Cospans and factorisation system.}
  In any category in which two classes of morphisms (containing all the
  isomorphisms) interact under pushouts as in the proposition, one can define a
  new category with the same objects, whose morphisms are isomorphism classes of
  cospans whose left leg belongs to one class and the right leg belongs to the
  other class.  Such cospans are composed by taking pushout.  By construction
  every morphism admits a factorisation as a morphism in the first class
  followed by a reversed-morphism in the second class.   This does {\em not}
  always constitute a factorisation system, since the factorisation is not
  necessarily unique up to {\em unique} comparison isomorphism, as required in a
  factorisation system.  However, if one of the two classes consists of epimorphisms,
  then the comparison isomorphism {\em will} be unique.
  
  % Example: if we take spans of arbitrary maps in $\Set$, composed by pullback,
  % then in the category of spans every map factors, but the comparison iso 
  % between isomorphic factorisations is not always unique.  Easiest example:
  % $ 1 \leftarro 2 \to 1$.
  
  % The isomoprhisms do not act freely on the generic morphisms at their codomain:
  % if $\graf{S}$ is an $\graf{n}$-graph which admits a nontrivial symmetries 
  % $j$ that  fixes the ports, then
  % the generic morphism $g: \graf{n} \to \graf{S}$ has an automorphism
  % $$
  % g = j \circ g
  % $$

  In the present case, the two classes are the reduced covers and the generic 
  morphisms.  Denote by $\kat{Cosp}$ the category
  whose objects are graphs without isolated edges, and whose morphisms are
  iso-classes of cospans in $\widetilde\Gr$ in which the first leg is a reduced cover
  and the second leg is a generic morphism
  $$\xymatrix @C+1.3em {
     \graf{T}\phantom{|}\ar[r]^-{\text{red.cover}} & \phantom{|}\graf{S}\phantom{|} & 
     \ar[l]_-{\text{refinement}} \phantom{|}\graf{R}   .
  }$$
  Since reduced covers 
  are epimorphisms in the Kleisli category (\ref{lem:epi}), we thus have:
\end{blanko}

\begin{prop}
  The category $\kat{Cosp}$ has a factorisations system in which the left-hand
  class consists of the cospans `reduced-cover / identity-read-backwards' and
  the right-hand class consists of `identity / generic-morphisms-backwards'.
\end{prop}

\begin{prop}\label{prop:commlawcommlaw}
  The commutation law in $\widetilde\Gr$ given by pushout is compatible
  with the commutation law in $\kat{BM}$ given by the factorisation system 
  (\ref{BMcommlaw}) via the equivalences $\kat{BM}_{\text{\rm gr}} \isopil \widetilde\Gr_{\text{\rm rc}}$
  and $\kat{BM}_{\text{\rm comp}} \isopil \widetilde\Gr\op_{\text{\rm gen}}$.
\end{prop}
\begin{proof}
  The statement is that given $h$ and $k$ in the commutation square in 
  \ref{BMcommlaw}, applying the two functors yields a pushout square as in 
  Proposition~\ref{prop:pushout}.  But this is clear from the explicit 
  descriptions given in \ref{BMcommlaw} and \ref{prop:pushout}.
\end{proof}

\begin{theorem}\label{thm:main}
  The reduced-cover/generic cospan category $\kat{Cosp}$ is equivalent to the
  Borisov--Manin category $\kat{BM}$.
\end{theorem}

\begin{proof}
  The argument is a general fact about gluing functors along commutation laws.
  We already have functors $\Phi_1: \kat{BM}_{\text{\rm gr}} \to
  {\widetilde\Gr}_{\text{\rm rc}}$ and $\Phi_2: \kat{BM}_{\text{\rm comp}} \to
  \widetilde\Gr\op_{\text{\rm gen}}$ which agree on objects and on isomorphisms.
  We define a functor $\Phi:\kat{BM} \to \kat{Cosp}$ on morphisms by the
  assignment
  $$
  \Phi(h) = \Phi_2(c) \circ \Phi_1(g)
  $$
  with reference to the factorisation $h = c \circ g$ of 
  Proposition~\ref{prop:fact}:
  $$ 
  \big(\xymatrix{ \tau \ar[r]^g & \sigma \ar[r]^c & \rho}\big) 
  \qquad \longmapsto \qquad
  \big(\xymatrix{ \graf{T} \ar[r]^{\Phi_1(g)} & \graf{S} & \ar[l]_{\Phi_2(c)} 
  \graf{R} }\big). 
  $$
  The factorisation $h = c \circ g$ is only determined up to isomorphism, but
  that is good enough, because the morphisms in $\kat{Cosp}$ are iso-classes of
  cospans.  This assignment respects composition by
  Proposition~\ref{prop:commlawcommlaw}, hence $\Phi$ is a functor.  It is clear
  that $\Phi$ agrees with $\Phi_1$ and $\Phi_2$ on the subcategories, and it
  induces an equivalence at the level of maximal groupoids (this follows from
  either Proposition~\ref{prop:JKrc=BMgr} and
  Proposition~\ref{prop:JKgen=BMcomp}).  In particular it is essentially
  surjective.  Finally is follows from 
  Proposition~\ref{prop:JKrc=BMgr} and
  Proposition~\ref{prop:JKgen=BMcomp} that $\Phi$ is
  fully faithful:
  \begin{align*}
  \Hom_{\kat{BM}}(\tau,\rho) 
  &= \int^{\sigma\in\operatorname{Iso}(\kat{BM})} \Hom_{\kat{BM}_{\text{\rm gr}}}(\tau,\sigma)\times
  \Hom_{\kat{BM}_{\text{\rm comp}}}(\sigma,\rho)
\\
& = \int^{\graf{S}\in\operatorname{Iso}\widetilde\Gr} \Hom_{\widetilde{\Gr}_{\text{\rm rc}}}(\graf{T},\graf{S}) 
  \times \Hom_{\widetilde{\Gr}_{\text{\rm gen}}}(\graf{R},\graf{S}) 
  \\[2pt]
  &= \Hom_{\kat{Cosp}}(\graf{T},\graf{R}).
  \end{align*}
  Here the integral signs are coends (see \cite[\S IX.6]{MacLane:categories}),
  which is the appropriate way of summing over all factorisations modulo
  isomorphisms at the middle object.
\end{proof}

\begin{blanko}{Remark.}
  The BM category of graphs plays an important role in the work of Kaufmann and
  Ward~\cite{Kaufmann-Ward:1312.1269} as source of examples of {\em Feynman
  categories}, an interesting alternative formalism to coloured operads (the two
  notions have been established equivalent in
  \cite{Batanin-Kock-Weber:1510.08934}).  In that context, the interest is in
  the full subcategory of the BM category spanned by the disjoint unions of
  corollas.  They show that every such morphism $\tau\to \rho$ is specified by a
  so-called ghost graph.  Their construction is clarified by the observation
  that their ghost graph is precisely the graph $\sigma$ appearing in the
  grafting-compression factorisation, or equivalently, is the apex of the 
  corresponding cospan.
%   Kaufmann and Ward define also ghost graphs for general morphisms in the BM
%   category.  In general their definition does not coincide with the
%   factorisation of the morphism into grafting-compression.
\end{blanko}

\medskip

\noindent {\bf Acknowledgments.} I thank Ralph Kaufmann for %patiently
explaining the Borisov--Manin graph category to me, some years ago.
I first met Nils Baas at the IMA in Minneapolis in 2004, and over the years
I have learned a lot from many discussions with him on graphs and higher structures.
I humbly dedicate this little piece to him at the occasion of his 70th 
birthday.  I regret I could not participate in the birthday conference.
This work has been supported by grant number MTM2013-42178-P of Spain.

\hyphenation{mathe-matisk}

% \bibliographystyle{scplain}
% \bibliography{joachims}

\noindent
Joachim Kock \texttt{<kock@mat.uab.cat>}\\
Departament de matem\`atiques\\
Universitat Aut\`onoma de Barcelona\\
08193 Bellaterra (Barcelona)\\
SPAIN.

\end{document}